\documentclass{article}
\usepackage{amsmath,amsfonts,amsthm,amssymb,mathrsfs,latexsym,paralist}

\usepackage{tikz}
\usetikzlibrary{arrows,automata}

\tikzstyle{V}=[fill=black,circle,scale=0.1, outer sep = 4pt]

\newtheorem{thm}{Theorem}[section]
\newtheorem{prop}[thm]{Proposition}
\newtheorem{cor}[thm]{Corollary}
\newtheorem{lemma}[thm]{Lemma}
\theoremstyle{remark}
\newtheorem{rmk}[thm]{Remark}

\theoremstyle{definition}
\newtheorem{defn}[thm]{Definition}

\DeclareMathOperator{\Ad}{Ad}

\newcommand{\z}{^{(0)}}
\newcommand{\2}{^{(2)}}
\newcommand{\inv}{^{-1}}
\newcommand{\1}{\textbf{1}}

\newcommand{\bi}{\begin{itemize}}
\newcommand{\ei}{\end{itemize}}
\newcommand{\be}{\begin{enumerate}}
\newcommand{\ee}{\end{enumerate}}

\newcommand{\T}{\mathbb{T}}
\renewcommand{\H}{\mathcal{H}}

\newcommand{\G}{\mathcal{G}}
\newcommand{\K}{\mathcal{K}}

\newcommand{\R}{\mathbb{R}}
\newcommand{\N}{\mathbb{N}}
\newcommand{\Z}{\mathbb{Z}}

\begin{document}

\title{$K$-theory and Homotopies of 2-Cocycles on Higher-Rank Graphs}
%
 \author{Elizabeth Gillaspy\thanks{Department of Mathematics, University of Colorado - Boulder, Campus Box 395, Boulder, CO 80309-0395, {\tt elizabeth.gillaspy@colorado.edu}}}
\maketitle

\begin{abstract}
This paper continues our investigation into the question of when a homotopy 
of 2-cocycles on a locally compact Hausdorff groupoid gives rise to an isomorphism of the $K$-theory groups of the twisted groupoid $C^*$-algebras.
Our main result, which builds on  work by Kumjian, Pask, and Sims, shows that a homotopy of 2-cocycles on a row-finite higher-rank graph $\Lambda$ gives rise to twisted groupoid $C^*$-algebras with isomorphic $K$-theory groups.  (Here, the groupoid in question is the path groupoid of $\Lambda$.) 
We also establish a technical result: any homotopy of 2-cocycles on a locally compact Hausdorff groupoid $\G$ gives rise to an upper semicontinuous bundle of $C^*$-algebras. 


\end{abstract}

\section{Introduction}


Higher-rank graphs, or $k$-graphs, provide a $k$-dimensional analogue of directed graphs.  They were introduced by Kumjian and Pask in \cite{kp} 
to provide a combinatorial model for the higher-rank Cuntz-Krieger algebras studied by Robertson and Steger in \cite{robertson-steger}. 
Much of the interest in the $C^*$-algebras $C^*(\Lambda)$ associated to  $k$-graphs $\Lambda$ stems from the multiple ways one can model $C^*(\Lambda)$ --- the  $k$-graph $\Lambda$ reflects many of the properties of $C^*(\Lambda)$, but we can also describe $C^*(\Lambda)$ as a universal $C^*$-algebra for certain generators and relations, or as a groupoid $C^*$-algebra $C^*(\Lambda) \cong C^*(\G_\Lambda)$.

The class of groupoids  includes groups, group actions, equivalence relations, and group bundles.  Renault initiated the study of groupoid $C^*$-algebras in \cite{renault}, and the theory and applications of groupoid $C^*$-algebras have since been developed by many researchers.
  Given a 2-cocycle $\omega \in Z^2(\G, \T)$ on a groupoid $\G$, Renault also explains in \cite{renault} how to construct the twisted groupoid $C^*$-algebra $C^*(\G, \omega)$.  
  These objects have received relatively little attention until quite recently, but it has now become clear that twisted groupoid $C^*$-algebras can help answer many questions about  the structure of untwisted groupoid $C^*$-algebras (cf.~\cite{cts-trace-gpoid-II, cts-trace-gpoid-III, clark-aH, dd-fell, jon-astrid}), as well as classifying those $C^*$-algebras which admit diagonal subalgebras (also known as Cartan subalgebras) --- cf.~\cite{c*-diagonals}. 
    In another direction, \cite{TXLG} establishes a connection between the $K$-theory of twisted groupoid $C^*$-algebras and the classification of $D$-brane charges in string theory. 
  
  
Two recent papers have explored the effect of a homotopy $\{\omega_t\}_{t\in[0,1]}$ of 2-cocycles on the $K$-theory of the twisted groupoid $C^*$-algebras.  First,  Echterhoff, L\"uck, Phillips, and Walters showed in  Theorem 1.9 of \cite{ELPW} that if $G$ is a group that satisfies the Baum-Connes conjecture with respect to the coefficient algebras $\K$ and $C([0,1],\K)$, and $\{\omega_t\}_{t \in [0,1]}$ is a homotopy of 2-cocycles on $G$, then the $K$-theory groups of the reduced twisted group $C^*$-algebras are unperturbed by the homotopy:
\begin{equation}
K_*(C^*_{r}(G, \omega_0)) \cong K_*(C^*_{r}(G, \omega_1)).
\end{equation}
In particular, taking $G = \Z^2$, we obtain another proof of the fact, established in 1980 by Pimsner and Voiculescu in \cite{pims-voic}, that all of the rotation algebras $\{A_\theta\}_{\theta \in [0,1]}$ have isomorphic $K$-theory groups.  

 
  Kumjian, Pask, and Sims also studied the effect of a homotopy of 2-cocycles on $K$-theory in \cite{kps-kthy}.  Theorem 5.4 of \cite{kps-kthy} establishes that if $\Lambda$ is a row-finite source-free $k$-graph and $c$ is a 2-cocycle on $\Lambda$ such that $c(\lambda, \mu) = e^{2\pi i \sigma(\lambda,\mu)}$ for some $\R$-valued 2-cocycle $\sigma$, then $K_*(C^*(\Lambda)) \cong K_*(C^*(\Lambda, c)).$  Defining $c_t(\lambda, \mu) = e^{2\pi i t \sigma(\lambda, \mu)}$ for $t \in [0,1]$ gives us a homotopy of 2-cocycles linking $c$ and the trivial 2-cocycle.  Moreover,
Corollary 7.8 of \cite{kps-twisted} tells us that $C^*(\Lambda, c)$ is isomorphic to a twisted groupoid $C^*$-algebra $C^*(\G_\Lambda, \omega_c)$.  Thus, we can view \cite{kps-kthy} Theorem 5.4 as a result about homotopic 2-cocycles on groupoids.

Inspired by the above-mentioned results, we have begun exploring the question of when a homotopy of 2-cocycles on a locally compact Hausdorff groupoid $\G$ induces an isomorphism of the $K$-theory groups of the (full or reduced) twisted groupoid $C^*$-algebras.
In a previous article \cite{transf-gps}, we extended the above-mentioned Theorem 1.9 of the paper \cite{ELPW} by Echterhoff et al.~to the case when $\G= G \ltimes X$ is a transformation group, where $X$ is locally compact Hausdorff and $G$ satisfies  the Baum-Connes conjecture with coefficients.

In this article, we prove the following generalization of \cite{kps-kthy} Theorem 5.4: 
\newtheorem*{main}{Theorem \ref{main}}
\begin{main}
Let $\Lambda$ be a row-finite $k$-graph with no sources and let $\{c_t\}_{t \in [0,1]}$ be a homotopy of 2-cocycles in $\underline{Z}^2(\Lambda, \T)$.  Then $\{c_t\}_{t\in [0,1]}$ gives rise to a homotopy $\{\sigma_{c_t}\}_{t \in [0,1]}$ of 2-cocycles on $\G_\Lambda$ such that
\[K_*(C^*(\G_\Lambda, \sigma_{c_0})) \cong K_*(C^*(\G_\Lambda, \sigma_{c_1})).\]
\end{main}

As of this writing, we are unaware of any examples of groupoids $\G$ and homotopies $\omega = \{\omega_t\}_{t\in[0,1]}$ of 2-cocycles on $\G$ where the homotopy does not induce an isomorphism of the $K$-theory groups of the twisted groupid $C^*$-algebras. 

\subsection{Outline}
This paper begins by recalling the definitions of a higher-rank graph and of a groupoid in Section \ref{2}, as well as the definition of a 2-cocycle in each category, and sketching the procedure by which we can construct a $C^*$-algebra from these objects.  In Section \ref{3} we define a homotopy of 2-cocycles on a $k$-graph and on a groupoid, and show that the definitions are compatible.  We also prove a technical result (Theorem \ref{fiber-alg}), namely, that a homotopy $\{\omega_t\}_{t \in [0,1]}$ of 2-cocycles on a groupoid $\G$ gives rise to a $C([0,1])$-algebra with fiber algebra $C^*(\G, \omega_t)$ at $t \in [0,1]$.  We expect that this result will prove useful in future work, as we search for more classes of groupoids where a homotopy of 2-cocycles induces an isomorphism of the $K$-theory groups of the twisted groupoid $C^*$-algebras.

In Section \ref{final} we begin the proof of Theorem \ref{main}.  Our proof technique consists of proving a stronger version of Theorem \ref{main} in a simple case, and then showing how to use this simple case to obtain our desired result for general $k$-graphs.  To be precise, Proposition \ref{trivial} shows that when the degree map $d$ on $\Lambda$ satisfies $d(\lambda) = b(s(\lambda)) - b(r(\lambda))$ for all $\lambda \in \Lambda$, the $C([0,1])$-algebra associated to a homotopy of 2-cocycles on $\Lambda$ is actually a trivial continuous field.   Section \ref{5} shows how to   exploit the triviality of the continuous field in this special case to see that a homotopy of 2-cocycles on any row-finite, source-free $k$-graph $\Lambda$ induces an isomorphism $K_*(C^*(\Lambda, c_0)) \cong K_*(C^*(\Lambda, c_1))$.  The argument in this section closely parallels Section 5 of \cite{kps-kthy}.

\textbf{Acknowledgments:} This research constitutes part of my PhD thesis, and I would like to thank my advisor, Erik van Erp, for his encouragement and assistance.  I would also especially like to thank Prof.~Jody Trout for his help and support throughout my PhD work.  

\section{Groupoids and $k$-graphs} 
\label{2}

\begin{defn}[cf.~\cite{kp} Definition 1.1]
\label{defn-kgraph}
A \emph{higher-rank graph} of degree $k$, or a $k$-graph, is a nonempty countable small category $\Lambda$ equipped with a functor $d: \Lambda \to \N^k$ (the \emph{degree map}) satisfying the following \emph{factorization property}: Given a morphism $\lambda \in \Lambda$ with $d(\lambda) = m + n$, there exist unique $\mu, \nu \in \Lambda$ such that $\lambda = \mu \nu$ and $d(\mu) = m, d(\nu) = n$.  \end{defn}
The simplest example of a $k$-graph is $\N^k$, equipped with the identity morphism $id: \N^k \to \N^k$.

In this article we will use the arrows-only picture of category theory, so that we think of the objects of a category $\Lambda$ as identity morphisms.  Hence, $\lambda \in \Lambda$ means that $\lambda$ is a morphism in $\Lambda$.  
Given an element $\lambda $ in a category $ \Lambda$, write $s(\lambda)$ for the domain, or \emph{source}, of the morphism $\lambda$, and $r(\lambda)$ for its target, or \emph{range}.
We say $k$-graph $\Lambda$ is \emph{row-finite} if, for any $v \in \text{Obj}(\Lambda)$ and any $n \in \N^k$, the set 
\[v \Lambda^n := \{ \lambda \in \Lambda: r(\lambda) = v, d(\lambda) = n\}\]
is finite.  We say $\Lambda$ has \emph{no sources} if $v \Lambda^n \not= \emptyset$ for every $v \in \text{Obj}(\Lambda)$ and every $n \in \N^k$.
We will only consider $k$-graphs which are row-finite and have no sources, since these are the $k$-graphs which we can study via the groupoid method that was introduced in \cite{kp} and which we will explain in Section \ref{groupoids}.


\begin{defn}[cf.~\cite{renault} Definition I.1.12, \cite{kps-twisted} Section 3]
For a category $\Lambda$, let $\Lambda^{*2} = \{ (\lambda_1,  \lambda_2) \in \Lambda \times \Lambda: s(\lambda_1) = r(\lambda_2)\}.$
A function $c: \Lambda^{*2} \to \T$ is called a \emph{2-cocycle} on $\Lambda$ if
\begin{equation}
c(\lambda, \mu\nu) c(\mu, \nu) = c(\lambda \mu, \nu) c(\lambda, \mu)
\label{k-cocycle}
\end{equation}
whenever $(\lambda, \mu), (\mu,\nu) \in \Lambda^{*2}$, and $c(\lambda, s(\lambda)) = c(r(\lambda), \lambda) = 1 \ \forall \ \lambda \in \Lambda.$
We write $\underline{Z}^2(\Lambda, \T)$ for the set of 2-cocycles on $\Lambda.$
\label{k-cocycle-defn}

If $c, \tilde{c}$ are two 2-cocycles on $\Lambda$, we say that $c,\tilde{c}$ are \emph{cohomologous} if there exists a function $b: \Lambda \to \T$ such that 
\[\tilde{c}(\mu,\nu) := b(\mu) b(\nu) b(\mu \nu)\inv c(\mu,\nu) = \delta b(\mu,\nu) c(\mu,\nu) \ \forall \ (\mu,\nu) \in \Lambda^{*2}.\]
We note that cohomologous 2-cocycles give rise to isomorphic twisted $C^*$-algebras (cf.~\cite{kps-twisted} Proposition 5.6, \cite{renault} Proposition II.1.2).
\end{defn}
The only cocycles we will consider  in this paper are 2-cocycles, so we will occasionally drop the 2 and refer to them simply as \emph{cocycles}.


%
%

\begin{defn}[\cite{kps-twisted} Definition 5.2]
The \emph{twisted higher-rank-graph algebra} $C^*(\Lambda, c)$ associated to a $k$-graph $\Lambda$ and a 2-cocycle $c$ on $\Lambda$ is the universal $C^*$-algebra generated by a collection $\{s_\lambda\}_{\lambda \in \Lambda}$ of partial isometries satisfying the following \emph{twisted Cuntz-Krieger relations}:
\bi
\item[(CK1)] $\{s_v\}_{v \in \text{Obj}(\Lambda)}$ is a collection of mutually orthogonal projections;
\item[(CK2)] $s_\mu s_\nu = c(\mu,\nu) s_{\mu \nu}$ whenever $s(\mu) = r(\nu)$;
\item[(CK3)] $s_\mu^* s_\mu = s_{s(\mu)}$ for all $\mu \in \Lambda$;
\item[(CK4)] $s_v = \sum_{\mu \in v\Lambda^n} s_\mu s_\mu^*$ for all $v \in \text{Obj}(\Lambda)$ and all $n \in \N^k$.
\ei
\end{defn}


Note that every $k$-graph $\Lambda$ admits at least one 2-cocycle: the trivial cocycle, obtained by setting $c(\lambda, \mu) = 1$ for all $(\lambda, \mu) \in \Lambda^{*2}$.
In this case, the definition above of $C^*(\Lambda, c)$ agrees with that of $C^*(\Lambda)$ given in \cite{kp} Definitions 1.5.  
For example, if $\Lambda = \N^k$ and $c$ is the trivial cocycle, then $C^*(\Lambda, c) \cong C(\T^k)$.
More generally, if $\Lambda = \N^2$, let $c_\theta: \Lambda^{*2} \to \T$ be given by $c_\theta( (m,n), (j,k)) = e^{2\pi i \theta nj}.$
Then $c_\theta$ is a 2-cocycle on $\Lambda$ and $C^*(\Lambda, c_\theta)$ is isomorphic to the rotation algebra $A_\theta$. 


\subsection{Groupoids}
\label{groupoids}
%

In this section, we review the construction of a twisted groupoid $C^*$-algebra set forth in \cite{renault}, as well as the procedure given in the seminal article \cite{kp} for associating a groupoid to a $k$-graph.  Theorem \ref{fiber-alg} in Section \ref{3} 
applies to arbitrary locally compact Hausdorff groupoids, so we present in full generality all the definitions necessary for the construction of a twisted groupoid $C^*$-algebra.  

A \emph{groupoid} $\G$ is a small category with inverses.  We use the notation of \cite{renault} to denote groupoid elements and operations; for example, $\G\2\subseteq \G \times \G$ denotes the set of composable pairs and $\G\z$ denotes the unit space. If $u \in \G\z$, write 
\[\G_u = \{x \in \G: s(x) = u\} \qquad \quad \G^u = \{x \in \G: r(x) = u\}.\]
 In this article, we restrict our attention to groupoids which admit a locally compact Hausdorff topology in which the operations of composition (or multiplication) and inversion are continuous.

In addition to the groupoids associated to $k$-graphs, examples of groupoids include groups, vector bundles, and transformation groups.  For more details and examples, see \cite{geoff-thesis, muhly-gpoids}.

Given a row-finite, source-free $k$-graph $\Lambda$, Section 2 of \cite{kp} describes how to form the associated path groupoid $\G_\Lambda$:

\begin{defn}[\cite{kp} Examples 1.7(ii)]
We define the $k$-graph $\Omega_k$ to be the category with $\text{Obj}(\Omega_k) = \N^k$, and morphisms $\Omega_k = \{(m,n) \in \N^k \times \N^k: n \geq m\}.$
We have $r(m,n) = m, s(m,n) = n, d(m,n) = n-m$.  Composition in $\Omega_k$ is given by  $(m,n)(n, \ell) = (m, \ell).$
\label{omega-k}
\end{defn}
For a $k$-graph $\Lambda,$ let $\Lambda^\infty$ denote the set of degree-preserving functors $x: \Omega_k \to \Lambda$.  When $k=1$, the elements $x \in \Lambda^\infty$ are the infinite paths in $\Lambda$. 

Given $p \in \N^k$, define $\sigma^p: \Lambda^\infty \to \Lambda^\infty$ by 
$\sigma^p(x)(m,n) = x(m+p, n+p).$
When $\Lambda$ is row-finite and source-free, Proposition 2.3 in \cite{kp} shows that if $\lambda \in \Lambda, x \in \Lambda^\infty$ satisfy $s(\lambda) = x(0)$, there is a unique $y \in \Lambda^\infty$ such that $\sigma^{d(\lambda)}(y) = x$; we often write $y=\lambda x$.

\begin{defn}[\cite{kp} Definition 2.1]
Given a row-finite, source-free $k$-graph $\Lambda$, the associated path groupoid $\G_\Lambda$ is the groupoid associated to the equivalence relation on $\Lambda^\infty$ of ``shift equivalence with lag.'' 
In other words,
\[\G_\Lambda := \{ (x, n-m,  y) \in \Lambda^\infty \times \Z^k  \times \Lambda^\infty: n,m \in \N^k, 
\sigma^n(x) = \sigma^m(y)\},\]
and $\G\z_\Lambda = \Lambda^\infty$, with $r(x, \ell, y) = x,\, s(x, \ell, y) = y$, and 
multiplication and inversion in $\G_\Lambda$  given by $(x, \ell, y) (y, m, z) = (x, \ell + m, z); \ (x, \ell, y)\inv = (y, -\ell, x).$

When $\Lambda$ is a row-finite, source-free $k$-graph,  Proposition 2.8 in \cite{kp} tells us that the sets 
\[Z(\mu, \nu) := \{(\mu x, d(\mu) - d(\nu), \nu x): x(0) = s(\mu) = s(\nu)\}\]
form a basis of compact open sets for a locally compact Hausdorff topology on $\G_\Lambda$ (in fact, with this topology $\G_\Lambda$ is an ample \'etale groupoid).
\end{defn}

To build a $C^*$-algebra out of a groupoid $\G$ we will  start by putting a $*$-algebra structure on $C_c(\G)$, and to do this we will need to integrate over the groupoid $\G$.  A Haar system $\{ \lambda^u\}_{u \in \G\z}$ (the groupoid analogue of Haar measure for groups, cf.~Definition I.2.2 in \cite{renault}) will allow us to do this.
Unlike in the group case, one cannot make existence or uniqueness statements about Haar systems for groupoids, so one usually starts by hypothesizing the existence of a fixed Haar system. 
For example, we obtain a Haar system $\{\lambda^x\}_{x \in \Lambda^\infty}$ on $\G_\Lambda$ by setting 
\[ \lambda^x(E) = \# \{ e \in E: e = (x, n, y) \text{ for some } n \in \Z^k, y \in \Lambda^\infty\}.\]
 We will always use this Haar system on $\G_\Lambda$ in this paper.

\begin{defn}
Let $\G$ be a locally compact Hausdorff groupoid equipped with a Haar system $\{\lambda^u\}_{u \in \G^{(0)}}$ and a continuous 2-cocycle $\omega$.  We define a $*$-algebra structure on $C_c(\G)$ as follows: for $f, g \in C_c(\G)$ let
\[f*_\omega g(a) = \int_{\G_{s(a)}} f(ab) g(b^{-1}) \omega(ab, b^{-1}) d\lambda^{s(a)}(b),\quad f^*(a) = \overline{f(a^{-1}) \omega(a,a^{-1})}.\]
\label{conv-alg}
\end{defn}
One can check (cf. \cite{renault} Proposition II.1.1)
that the multiplication is well defined (that is, that $f*_\omega g \in C_c(\G)$ as claimed) and associative, and that $(f^*)^* = f$ so that the involution is involutive.  The proof of associativity relies on the cocycle condition \eqref{k-cocycle}.

Given the fundamental role that the cocycle $\omega$ plays in the multiplication and involution on $C_c(\G)$, we will often write $C_c(\G, \omega)$ to denote the set $C_c(\G)$ equipped with the $*$-algebra structure of Definition \ref{conv-alg}.
We define $C^*(\G, \omega)$ to be the completion of $C_c(\G, \omega)$ in the maximal $C^*$-norm, as described in Chapter II of \cite{renault}. 

\begin{defn}
When $\G= \G_\Lambda$ is the groupoid associated to a row-finite $k$-graph $\Lambda$ with no sources, Lemma 6.3 of \cite{kps-twisted} explains how, given a cocycle $c \in \underline{Z}^2(\Lambda, \T)$, we can construct a cocycle $\sigma_c \in Z^2(\G_\Lambda, \T)$.  Then Corollary 7.8 of \cite{kps-twisted} shows that $C^*(\Lambda, c) \cong C^*(\G_\Lambda, \sigma_c)$.  
 The construction of $\sigma_c$ is rather technical, but since we will need the details later, we present it here.
 
 Lemma 6.6 of \cite{kps-twisted} establishes the existence of a subset 
 \[\mathcal{P} \subseteq \{Z(\mu, \nu): s(\mu) = s(\nu)\}\] 
 that partitions $\G_\Lambda$.  In other words, every $a \in \G_\Lambda$ has exactly one representation of the form $a = (\mu_a x, d(\mu_a) -d(\nu_a),  \nu_a x)$ with $Z(\mu_a, \nu_a) \in  \mathcal{P}$.  Note that if $(a,b) \in \G\2_\Lambda$, we need not have $\mu_a = \mu_{ab}$ or $\nu_b = \nu_{ab}$.  However, given $(a,b) \in \G\2_\Lambda$, Lemma 6.3(i) of \cite{kps-twisted} shows that we can always find $y \in \Lambda^\infty$ and $\alpha, \beta, \gamma \in \Lambda$ such that 
 \begin{align*}
 a&= (\mu_a \alpha y, d(\mu_a) - d(\nu_a), \nu_a \alpha y)\\
 b&= (\mu_b \beta y, d(\mu_b) - d(\nu_b), \nu_b \beta y) \\
 ab&= (\mu_{ab} \gamma y, d(\mu_{ab}) - d(\nu_{ab}), \nu_{ab} \gamma y).
 \end{align*}
Then, given a 2-cocycle $c$ on $\Lambda$, we define a 2-cocycle $\sigma_c$ on $\G_\Lambda$ by
\[\sigma_c(a, b) = c(\mu_a, \alpha) c(\mu_b, \beta) c(\nu_{ab}, \gamma) \overline{ c(\nu_a, \alpha) c(\nu_b, \beta) c(\mu_{ab}, \gamma) }.\]
Since $c$ satisfies the cocycle condition \eqref{k-cocycle}, it's straightforward to check that $\sigma_c$ does also.
Lemma 6.3 of \cite{kps-twisted} checks that $\sigma_c$ is well-defined 
and continuous, so we can construct the groupoid $C^*$-algebra $C^*(\G_\Lambda, \sigma_c)$ as outlined above.  Corollary 7.8 of \cite{kps-twisted} tells us that $C^*(\G_\Lambda, \sigma_c) \cong C^*(\Lambda, c)$.
\label{sigma-c}
\end{defn}

Theorem 6.5 of \cite{kps-twisted} establishes that different choices of partitions $\mathcal{P}$ will give rise to cohomologous groupoid cocycles, and hence to isomorphic twisted groupoid $C^*$-algebras. 


\section{Homotopies of Cocycles}
\label{3}

In order to define a homotopy of groupoid 2-cocycles, we begin by observing that, given any locally compact Hausdorff groupoid $\G$, we can make $\G \times [0,1]$ into a LCH groupoid by equipping it with the product topology and setting $(\G \times [0,1])\2 := \G\2 \times [0,1]$.  
In other words,  $(\G \times [0,1])\z = \G\z \times [0,1]$, and 
\[r(\gamma, t) = (r(\gamma), t), \ s(\gamma, t) = (s(\gamma), t).\]
Moreover, if $\G$ has a Haar system $\{\lambda^u\}_{u \in \G\z}$, then setting $\lambda^{u,t} := \lambda^u$ for every $t \in [0,1]$ gives rise to a Haar system on $\G \times [0,1]$.  We will always use this Haar system on $\G \times [0,1]$ in this paper.

\begin{defn}[\cite{transf-gps} Definition 2.12] A \emph{homotopy of (2-)cocycles} on a locallpy compact Hausdorff groupoid $\G$ is a 2-cocycle $\omega \in Z^2(\G \times [0,1], \T)$.  We say that two cocycles $\omega_0, \omega_1 \in Z^2(\G, \T)$ are \emph{homotopic} if there exists a homotopy $\omega \in Z^2(\G \times [0,1], \T)$ such that $\omega_i = \omega|_{\G \times \{i\}}$ for $i = 0,1$.
\label{htopy}
\end{defn}

%

If $\omega$ is a homotopy of cocycles on $\G$ linking $\omega_0, \omega_1$,  Theorem \ref{fiber-alg} below tells us that  $C^*(\G, \omega_0)$ and $C^*(\G, \omega_1)$ are quotients of $C^*(\G \times [0,1], \omega)$.  This 
will be fundamental to the proof of our main result, Theorem \ref{main}. 

%
\begin{defn}[cf. \cite{xprod} Definition C.1]
Let $X$ be a locally compact Hausdorff space.  A $C^*$-algebra $A$ is a \emph{$C_0(X)$-algebra} if we have a $*$-homomorphism $\Phi: C_0(X) \to ZM(A)$ 
such that 
\[ A = \overline{\text{span}} \{ \Phi(f) a: f \in C_0(X), a \in A\} .\]  
We usually write $f \cdot a$ for $\Phi(f)a$.
\end{defn}
If $A$ is a $C_0(X)$-algebra,  then for any $x \in X$, $\overline{\text{span}}  C_0(X \backslash x) \cdot A$ is an ideal $I_x$.  We call $A_x := A/I_x$ the \emph{fiber} of $A$ at $x \in X$.
%

\begin{thm} Let $\omega$ be a homotopy of cocycles on a locally compact Hausdorff groupoid $\G$ with Haar system $\{\lambda^u\}_{u \in \G\z}$. Then $C^*(\G \times [0,1], \omega)$ is a $C([0,1])$-algebra, with fiber  $C^*(\G, \omega_t)$ at $t \in [0,1]$.
\label{fiber-alg}
\end{thm}
\begin{proof}
We begin by checking that $C^*(\G \times [0,1], \omega)$ is a $C([0,1])$-algebra.
For $f \in C([0,1]), \ \phi \in C_c(\G \times [0,1], \omega)$, define 
\[f \cdot \phi(a,  t) = f(t) \phi(a,t).\]
It's not difficult to check that this action extends to a $*$-homomorphism 
\[\Phi: C([0,1]) \to ZM(C^*(\G \times [0,1], \omega))\] such that 
$||\Phi(f ) \phi|| \leq ||f||_\infty ||\phi||,$
 or that 
\[\Phi(C([0,1])) \cdot C_c(\G \times [0,1], \omega) = C_c(\G\times [0,1], \omega) \]
is dense in $C^*(\G \times [0,1], \omega)$.
In other words, $\Phi$ makes $C^*(\G \times [0,1], \omega)$ into a $C([0,1])$-algebra as claimed.



Fix $t \in [0,1]$ and denote by $q_t: C_c(G \times [0,1], \omega) \to C_c(G, \omega_t)$  the  evaluation map.  
Then $q_t$ is bounded by the $I$-norm (cf.~\cite{renault} Section II.1), and hence extends to a surjective $*$-homomorphism 
 $ q_t: C^*(\G \times [0,1], \omega) \to C^*(\G, \omega_t)$. 
 In other words, $C^*(\G, \omega_t)$ is a quotient of $C^*(\G \times [0,1], \omega)$.
To see that $C^*(\G, \omega_t) \cong C^*(\G \times [0,1], \omega)_t$, we need to check that $\ker q_t = I_t$.
A standard approximation argument will show that $\ker q_t \supseteq I_t$: 
thus, we will only detail the proof that $\ker q_t \subseteq I_t$.  



%
%
Note that the fiber algebra
$C^*(\G \times [0,1],\omega)_t \cong C^*(\G \times [0,1], \omega)/I_t$ can be calculated as a completion $\overline{C_c(\G \times [0,1], \omega)}$
 with respect to the norm given by
\[\|f\|_t := \sup \{ \|L(f)\|: L(I_t) = 0, \ L\text{ is an $I$-norm-bounded representation} \}.\]
Thus, to show that $\ker q_t \subseteq I_t$, we will show that each  such  representation $L$ factors through $ q_t $.



Given such a representation $L: C_c(\G \times [0,1], \omega) \to B(\H)$, define $L': C_c(\G, \omega_t) \to B(\mathcal{H})$ by $L'(q_t(f)) := L(f).$
We claim that  $L'$ is an $I$-norm-bounded representation of $C_c(\G, \omega_t)$.  To see this, it suffices
to check that $L'$ is well-defined and bounded. 


\begin{lemma}
 If $f, g \in C_c(\G \times [0,1], \omega)$ satisfy $q_t(f) = q_t(g)$, then the function $h = f -g \in C_c(\G  \times [0,1], \omega)$ lies in $I_t$.  Consequently, $L(f) = L(g)$ and $L'$ is well defined on $C_c(\G , \omega_t)$.
 \label{well-def}
\end{lemma}
\begin{proof}
Let $\{f_i\}_{i\in I}$ be an approximate unit for $C_0([0,1]\backslash t)$ such that $f_i(s) \nearrow 1$ for every $s \not= t$, and moreover that for each $i$ there exists $\delta_i > 0$ such that $f_i(s) =1$ if $|s-t| \geq \delta_i$.  We will show that the $I$-norm $\| h - f_i h\|_I \to 0$.  Consequently, $h = \lim_i f_i h $ in $C^*(\G \times [0,1], \omega)$, so $h\in I_t$.

For any $k \in C_c(\G \times [0,1], \omega)$, the axioms of a Haar system tell us that the function $(u, t) \mapsto \int |k(a, t) | d\lambda^{u,t}(a)$ is in $C_0(\G\z \times [0,1]).$ 
  In particular, if we take $k$ to be a function that equals 1 where $h$ is nonzero, and vanishes rapidly off {supp} $h$, this shows us that $\phi(u,t) := \lambda^{u,t}(\text{supp }h) $ is a pointwise limit of functions in $C_0(\G^{(0)}\times [0,1])$, and hence is bounded.  
   Let $K = \max \phi$.

Let $\epsilon > 0$ be given. Since $h$ is compactly supported and $h(a,t) = 0 \ \forall \ a \in \G$, we can choose $\delta >0$ such that $| h(a, s) |< \epsilon/K \ \forall \ a \in \G$ whenever $ |s-t| < \delta$, and choose $j$ such that $i \geq j$ means $\delta_i < \delta$.  Then, if $|s-t| < \delta$,
\begin{align*}
\int_{\G^u}| h(a, s) - f_i(s) h(a,s)| \, d\lambda^u(a) &= (1-f_i(s)) \int_{G^u} |h(a,s)| \, d\lambda^u(a)< 1 \cdot \epsilon.
\end{align*}
On the other hand, if $|s-t| \geq \delta > \delta_i$, then $f_i(s) = 1$ and 
\[\int_{\G^u}| h(a, s) - f_i(s) h(a,s)| \, d\lambda^u(a) = 0\]
for any $u \in \G\z$.  
In either case, given any $\epsilon > 0$ we can always choose $j$ such that $i \geq j$ implies
\begin{align*}
 \| h - f_i h\|_I & = \max \left\{\sup_{s\in [0,1]} \sup_{u \in \G\z} \int_{\G^u} | h(a, s) - f_i(s) h(a, s)| d\lambda^u(a), \right.\\
 & \qquad \qquad \left. \sup_{s \in [0,1]}\sup_{u \in \G\z} \int_{\G^u} | h(a\inv, s) - f_i(s) h(a\inv, s)| d\lambda^u(a)\right\} \\
 &< \epsilon.
 \end{align*}
Since $\| h - f_i h\| \leq \| h - f_i h \|_I$ and $I_t$ is closed, it follows that $h \in I_t$ as desired, and so $L(h) = 0$. \end{proof}

Having seen that $L'$ is well defined, we now proceed to show that it is bounded.

\begin{lemma}
For any fixed $f \in C_c(\G \times [0,1], \omega)$, the map $s \mapsto \|q_s(f)\|_I$ is continuous.
\end{lemma}
\begin{proof}
Fix $f \in C_c(\G \times [0,1], \omega)$ and fix $t \in [0,1]$.
As in the proof of Lemma \ref{well-def}, let $K$ denote the supremum of the function $(u,s) \mapsto \lambda^{u,s}(\text{supp } f).$
Since $f$ has compact support, given $\epsilon > 0$ we can choose $\delta$ such that 
\[|s-t| < \delta \Rightarrow | f(a,t) - f(a,s)| < \frac{\epsilon}{2K} \ \forall \ a \in \G.\]
Now, by definition of the $I$-norm, there exists $u \in \G\z$ such that either
\begin{align*}
\| q_s(f)\|_I &< \int_{\G^u} |f(a,s)|\, d\lambda^u(a) + \frac{\epsilon}{2},
\text{ or }  \| q_s(f)\|_I < \int_{\G^u} |f(a\inv,s)|\, d\lambda^u(a) + \frac{\epsilon}{2}.
\end{align*}
It follows that either
\begin{align*}
\| q_s(f)\|_I &< \int_{\G^u} |f(a, t)| + \frac{\epsilon}{2K}\, d\lambda^u(a) + \frac{\epsilon}{2}\leq \int_{\G^u} |f(a,t)| \, d\lambda^u(a) + \epsilon, \\
\text{ or} \quad \| q_s(f)\|_I &< \int_{\G^u} |f(a\inv, t)| + \frac{\epsilon}{2K}\, d\lambda^u(a) + \frac{\epsilon}{2}\leq \int_{\G^u} |f(a\inv,t)| \, d\lambda^u(a) + \epsilon.
\end{align*}
Thus,
\begin{align*}
\|q_s(f)\|_I &< \max \left \{ \int_{\G^u} |f(a,t)| \, d\lambda^u(a) , \int_{\G^u} |f(a\inv,t)| \, d\lambda^u(a) \right\} + \epsilon  \\
& \leq  \max \left \{\sup_{u\in \G\z}  \int_{\G^u} |f(a,t)| \, d\lambda^u(a) , \sup_{u\in \G\z}  \int_{\G^u} |f(a\inv,t)| \, d\lambda^u(a) \right\} + \epsilon \\
& = \|q_t(f)\|_I + \epsilon
\end{align*}
if $|s-t| < \delta$.  Reversing the roles of $s$ and $t$ in the above argument tells us that 
\[|s-t| < \delta \Rightarrow \left| \| q_s(f)\|_I - \|q_t(f)\|_I \right| < \epsilon,\]
as desired. \end{proof}

Now we can complete the proof of Theorem \ref{fiber-alg}. 
Set $S_t = \{\psi \in C\left([0,1]\right) : \psi(t) = 1\}$;  for any $\psi \in S_t$ and any $f \in C_c(\G \times [0,1], \omega)$, we have
\begin{equation*}
||L(\psi \cdot f) ||= ||L'(q_t (\psi \cdot f)) || = ||L'(q_t(f))||.
\end{equation*}
Consequently, 
\begin{align*}
||L'(q_t(f))|| = \inf_{\psi \in S} ||L(\psi \cdot f) || &\leq \inf_{\psi} ||\psi \cdot f||_{I} \\
&= \inf_\psi \max \left\{\sup_{s \in [0,1]} \sup_{u \in \G^{(0)}} \int | \psi(s) f(a,s)|  \, d \lambda^u (a), \right. \\
&\qquad \left. \sup_{s\in[0,1]}\sup_{u\in \G\z} \int | \psi(s) f(a\inv, s)| d\lambda^u(a) \right\} \\
&= \inf_\psi \sup_{s \in [0,1]} ||q_s( \psi \cdot f)||_{I}.
\end{align*}

Let $\epsilon >0$ be given.  Choose $\delta$ such that  $|s-t| < \delta \Rightarrow \left| \, \|q_s(f)\|_{I} - \|q_t(f)\|_{I} \right|< \epsilon;$
 choose  $\psi_\epsilon \in C([0,1])$ such that $\psi_\epsilon(t) = 1$ and $|s-t| \geq \delta \Rightarrow \psi_\epsilon(s) = 0.$   Then, since $\psi_\epsilon \in S_t$, 
\begin{equation}
 ||q_s(\psi_\epsilon \cdot f)||_{I}  = \psi_\epsilon(s) ||q_s(f)||_{I} < \psi_\epsilon(s) \left( ||q_t(f)||_{I} + \epsilon\right) \leq  ||q_t(f)||_{I} + \epsilon
 \label{1.5}
\end{equation}
if $|s-t| < \delta$; otherwise we have $||q_s(\psi_\epsilon \cdot f)||_{I} = 0$, and  \eqref{1.5} still holds.

Since we can find such a $\psi_\epsilon$ for any $\epsilon > 0$, it follows that
\begin{align*}
||L'(q_t(f))|| & \leq \inf_{\psi \in S_t} \sup_{s\in [0,1]} ||q_s( \psi \cdot f)||_{I} \leq \inf_\epsilon \sup_s ||q_s(\psi_\epsilon \cdot f) ||_{I} \\
& \leq \inf_\epsilon ||q_t(f)||_{I} + \epsilon \\
& = ||q_t(f)||_{I}.
\end{align*}
The fact that $q_t$ is onto now tells us that $L'$ is a bounded representation of $C_c(\G, \omega_t)$ as claimed.  In other words, every representation $L$ of $C_c(\G \times [0,1], \omega)$ that kills $I_t$ also factors through $q_t$, so $\ker q_t \subseteq I_t$.  
This completes the proof that the fiber algebra $C^*(\G \times [0,1], \omega)/I_t$ of the $C([0,1])$-algebra $C^*(\G \times [0,1], \omega)$ is simply $C^*(\G, \omega_t)$. \end{proof}

In order to apply Theorem \ref{fiber-alg} to a homotopy of cocycles on a $k$-graph, we first need to define such a homotopy.  Unlike for groupoids, there is no obvious way to make $\Lambda \times [0,1]$ into a higher-rank graph,  so our definition of a homotopy of $k$-graph cocycles will look rather different than Definition \ref{htopy} above.  However, Proposition \ref{2.1} below will show that the two definitions are compatible.

\begin{defn}
\label{cocycle-htopy}
Let $\Lambda$ be a $k$-graph. A family $\{c_t\}_{t \in [0,1]}$ of 2-cocycles in $\underline{Z}^2(\Lambda, \T)$ is a \emph{homotopy of (2-)cocycles} on $\Lambda$ if for each pair $(\lambda, \mu) \in \Lambda^{*2}$ the function 
$t \mapsto c_t(\lambda, \mu) \in \T$
is continuous.
\end{defn} 

\begin{defn}
Let $\{c_t\}_{t \in [0,1]}$ be a homotopy of cocycles on a $k$-graph $\Lambda$.  Define $\omega \in Z^2(\G_\Lambda \times [0,1], \T)$ by 
\[\omega\left( (a, t), (b,t) \right) = \sigma_{c_t}(a,b),\]
where $\sigma_{c_t}$ is the cocycle on $\G_\Lambda$ associated to $c_t$ as in Definition \ref{sigma-c}.  
\label{gpoid-htopy}
\end{defn}

A moment's thought will reveal that $\omega$ satisfies the cocycle condition \eqref{k-cocycle}, since each $\sigma_{c_t}$ is a cocycle.  Thus, in order to see that $\omega$ is a homotopy of cocycles on $\G_\Lambda$, we merely need to check that $\omega: (\G_\Lambda \times [0,1])\2 \to \T$ is continuous.

\begin{prop}
The cocycle $\omega$ described in Definition \ref{gpoid-htopy} is continuous, and hence is a  homotopy of groupoid cocycles on $\G_\Lambda$. 
\label{2.1}
\end{prop}
\begin{proof}
We will show that if $\{(a_i, b_i, t_i)\}_{i \in I} \subseteq \G\2_\Lambda \times [0,1]$ is a net which converges to $(a,b,t)$, then 
\begin{equation}
\omega\left( (a_i, t_i), (b_i, t_i) \right) := \sigma_{c_{t_i}}(a_i, b_i) = \sigma_{c_{t_i}}(a,b)
\label{idk}
\end{equation}
for large enough $i$.
Recall from Definition \ref{sigma-c} that $\sigma_{c_{t_i}}(a,b)$ is a finite product of terms of the form $c_{t_i}(\mu, \nu)$ and their inverses, where the elements $\mu, \nu$ depend only on the elements $a, b$ and on the choice of partition $\mathcal{P}$ of $\G_\Lambda$ -- but not on the 2-cocycle $c_{t_i}$.  Thus,  Equation \eqref{idk}, and the continuity of the maps $t \mapsto c_t(\mu, \nu)$, will imply that $\omega\left( (a_i, t_i), (b_i, t_i) \right) \to \sigma_{c_t}(a,b) = \omega\left( (a, t), (b,t) \right).$

In what follows, we will use the notation of Definition \ref{sigma-c}.  
If $(a_i, b_i ,t_i) \to (a,b,t)$, then for large enough $i$ we have $a_i \in Z(\mu_a, \nu_a)$, $b_i \in Z(\mu_b, \nu_b)$, and  
  $a_i b_i \in Z(\mu_{ab}, \nu_{ab})$ as well.  In other words, we can write 
\begin{align*}
a& = (\mu_a \alpha y,  d(\mu_a) - d(\nu_a), \nu_a \alpha y), \quad a_i  = (\mu_a \alpha_i y_i, d(\mu_a) - d(\nu_a), \nu_a \alpha_i y_i) \\ 
b&= (\mu_b \beta y, d(\mu_b) - d(\nu_b),  \nu_b \beta y) , \quad 
b_i = (\mu_b \beta_i y_i, d(\mu_b) - d(\nu_b), \nu_b \beta_i y_i)\\
ab &= (\mu_{ab} \gamma y,d(\mu_{ab}) - d(\nu_{ab}), \nu_{ab} \gamma y) , \quad 
a_ib_i = (\mu_{ab} \gamma_i y_i, d(\mu_{ab}) - d(\nu_{ab}),\nu_{ab} \gamma_i y_i)
\end{align*}
for some $\alpha, \beta, \gamma, \alpha_i, \beta_i , \gamma_i \in \Lambda$ and $y, y_i \in \Lambda^\infty$.

Since $a_i \to a$ we must also have $\alpha_i y_i \to \alpha y$ in $\Lambda^\infty$. Thus, for large enough $i$, $\alpha_i y_i \in Z(\alpha): = \{\alpha y: y \in \Lambda^\infty,\, y(0) = s(\alpha)\}$ (cf.~Proposition 2.8 of \cite{kp}).
It follows that  
 \begin{align*}
 a_i= (\mu_a \alpha y_i', d(\mu_{a}) - d(\nu_{a}), \nu_a \alpha y_i'), &\quad 
 b_i = (\mu_b \beta z_i', d(\mu_{b}) - d(\nu_{b}),\nu_b \beta z_i') \\
 a_ib_i = (\mu_{ab} \gamma w_i', d(\mu_{ab}) & - d(\nu_{ab}),\nu_{ab} \gamma w_i'),
 \end{align*}
where (since each pair $(a_i, b_i) \in \G_{\Lambda}\2$ by hypothesis)
\[\nu_a \alpha y_i' =  \mu_b \beta z_i'; \quad \mu_a \alpha y_i' = \mu_{ab} \gamma w_i'; \quad \nu_b\beta z_i' = \nu_{ab} \gamma w_i'.\]

Now, 
$\nu_a \alpha = \mu_b \beta$ by \cite{kps-twisted} Lemma 6.3, and thus $y_i' = z_i'$.  A similar argument gives $z_i' = w_i'$ as well, so $y_i' = z_i' = w_i'$. 
In other words, 
for large enough $i$, 
\begin{align*}
\sigma_{c_{t_i}}(a_i, b_i) &= c_{t_i}(\mu_a, \alpha) c_{t_i}(\mu_b, \beta) c_{t_i}(\nu_{ab}, \gamma) \overline{ c_{t_i}(\nu_a, \alpha) c_{t_i}(\nu_b, \beta) c_{t_i}(\mu_{ab}, \gamma) } \\
&= \sigma_{c_{t_i}}(a,b),
\end{align*}
as claimed.  As observed in the first paragraph of the proof, it now follows that $\omega$ is a homotopy of cocycles on $\G_\Lambda$ as desired. \end{proof}

\begin{cor} Let $\{c_t\}$ be a homotopy of cocycles on a $k$-graph $\Lambda$, and define a cocycle $\omega$ on $\G_\Lambda \times [0,1]$ as in Definition \ref{gpoid-htopy}.  Then $C^*(\G_\Lambda \times [0,1], \omega)$ is a $C([0,1])$-algebra with fiber algebra $C^*(\G_\Lambda, \sigma_{c_t}) \cong C^*(\Lambda, c_t)$ at $t \in [0,1]$. 
 \end{cor}
\begin{proof}
 Proposition \ref{2.1} tells us that $\omega$ is a homotopy of cocycles on $\G_\Lambda$, and Theorem \ref{fiber-alg} tells us that the fiber over $t \in [0,1]$ of the $C([0,1])$-algebra $C^*(\G_\Lambda \times [0,1], \omega)$  is  $C^*(\G_\Lambda, \sigma_{c_t}) $.  The final isomorphism is provided by Corollary 7.8 of \cite{kps-twisted}. \end{proof}
 
\section{The main theorem}
Our goal in this section is to prove the following:
\label{final}

\begin{thm}
Let $\Lambda$ be a row-finite $k$-graph with no sources and let $\{c_t\}_{t \in [0,1]}$ be a homotopy of cocycles on $\Lambda$.  Then 
\[K_*(C^*(\Lambda, c_0)) \cong K_*(C^*(\Lambda, c_1)).\]
Moreover, this isomorphism preserves the $K$-theory class of the vertex projection $s_v$ for each $v \in \text{Obj}(\Lambda)$.
\label{main}
\end{thm}

We begin by proving a stronger version of Theorem \ref{main} in the simpler case when the degree functor $d$ satisfies $d(\lambda) =  \delta b(\lambda):= b(s(\lambda)) - b(r(\lambda))$
 for some function $b: \text{Obj}(\Lambda) \to \Z^k$; this is Proposition \ref{trivial} below.  We then combine Proposition \ref{trivial} with techniques from \cite{kps-kthy} to prove Theorem \ref{main} in full generality.

\subsection{The AF case} 
 \label{4}

 If $(\Lambda, d)$ is a $k$-graph such that $d = \delta b$, then Lemma 8.4 of \cite{kps-twisted} tells us that $C^*(\Lambda, c)$ and $C^*(\Lambda)$ are both AF-algebras, with the same approximating subalgebras and multiplicities of partial inclusions.  Consequently, $C^*(\Lambda, c) \cong C^*(\Lambda)$.  In order to fix notation for what follows, we describe this isomorphism in some detail.
 
Lemma 3.1 of \cite{kp} shows that if $\Lambda$ is a row-finite, source-free $k$-graph, then  $\{s_\lambda s_\mu^*: s(\lambda) = s(\mu)\}$ spans a dense $*$-subalgebra of $C^*(\Lambda)$. Moreover, when $d = \delta b$, Lemma 5.4 of \cite{kp} tells us that 
$\{s_\lambda s_\mu^*: b(s(\lambda)) = b(s(\mu)) = n\}$
 forms a collection of matrix units for the  subalgebra 
\[A_n = \overline{\text{span}}\{s_\lambda s_\mu^*: b(s(\lambda)) = b(s(\mu)) = n\}\cong \bigoplus_{b(v) = n} \K(\ell^2(s\inv(v))).\]
Observe that we can think of $A_n$ as a subalgebra of $C^*(\Lambda)$ or of $C^*(\Lambda, c)$.  In fact, these subalgebras allow us to exhibit $C^*(\Lambda, c)$ and $C^*(\Lambda)$ as AF algebras:
\[C^*(\Lambda, c) = \varinjlim (A_n, \phi_{m,n}^c) \text{ and } C^*(\Lambda) = \varinjlim (A_n, \phi_{m,n}),\]
where the connecting maps $\phi_{m,n}, \phi_{m,n}^c: A_n \to A_m$ are given by
  \begin{align*} \phi^c_{m, n}(s_\lambda s_\mu^*) &= \sum_{r(\alpha) = s(\lambda), b(s(\alpha)) = m} c(\lambda, \alpha) \overline{c(\mu, \alpha)} s_{\lambda \alpha}s_{\mu \alpha}^*\\
\phi_{m,n}(s_{\lambda}s_\mu^*) &= \sum_{r(\alpha) = s(\lambda), b(s(\alpha)) = m} s_{\lambda \alpha}s^*_{\mu \alpha}.\end{align*}
%


We can now describe explicitly the isomorphism $C^*(\Lambda, c) \cong C^*(\Lambda)$.
 As in Theorem 4.2 of \cite{kps-kthy}, write  $\1$ for $(1, \ldots, 1) \in \N^k$, and define $\kappa: \Lambda \to \T$   by 
\[\kappa(\lambda) = \left\{ \begin{array}{cl} 1, & d(\lambda) \not\geq \1 \\ \kappa(\mu)c(\mu, \alpha), & d(\alpha) = \1 \text{ and } \lambda = \mu \alpha.\end{array}\right.\]
For $n \in \Z^k$, let $U_n = \sum_{b(s(\lambda)) = n} \kappa(\lambda) s_\lambda s_\lambda^* \in U(M(A_n)).$ 
A quick computation will show that for any $\lambda, \mu$ with $s_\lambda s_\mu^* \in A_n$,
\begin{equation}\Ad U_n(s_\lambda s_\mu^*) = \kappa(\lambda)\overline{ \kappa(\mu)} s_\lambda s_\mu^*.\label{U}
\end{equation} 
 Moreover, the factorization property tells us that for any $h \in \Z$,  
\[\phi^{c}_{(h+1)\1, h\1} \circ  \Ad U_{h\1} =\Ad U_{(h+1)\1} \circ \phi_{(h+1)\1, h\1} .\]
In other words,  $\Ad U_*$ intertwines the connecting maps $\phi_{m,n}^c, \phi_{m,n}$, and hence implements the isomorphism $C^*(\Lambda) \to C^*(\Lambda, c)$.
 
 We can now use this isomorphism to prove that a homotopy of cocycles on $\Lambda$ gives rise to a trivial continuous field when $d = \delta b$:
 
\begin{prop}
Let $(\Lambda,d)$ be a row-finite, source-free $k$-graph such that $d = \delta b$ for some function $b: \text{Obj}(\Lambda) \to \Z^k$; let $\{c_t\}_{t \in [0,1]}$  be a homotopy of cocycles on $\Lambda$; and let $\omega$ be the cocycle on $\G_\Lambda \times [0,1]$ associated to $\{c_t\}_{t \in [0,1]}$ as in Definition \ref{gpoid-htopy}.  We have an isomorphism of $C([0,1])$-algebras
\[C^*(\G_\Lambda \times [0,1], \omega) \cong C^*(\G_\Lambda\times [0,1])  \cong C([0,1]) \otimes C^*(\Lambda).
\]
\label{trivial}
\end{prop} 
\begin{proof}
Recall that 
\[C^*(\G_\Lambda \times [0,1])_t \cong C^*(\G_\Lambda, \sigma_{c_t}) \cong C^*(\Lambda, c_t) \cong C^*(\Lambda)\]
if $d = \delta b$.
Thus, the $C([0,1])$-algebras $C^*(\G_\Lambda \times [0,1], \omega)$ and $C([0,1]) \otimes C^*(\Lambda)$ have isomorphic fibers over each point $t \in [0,1]$.  

In order to prove the Proposition, we need to show that these isomorphisms $C^*(\G_\Lambda, \sigma_{c_t}) \cong C^*(\Lambda)$ vary continuously in $t$, so that they patch together to give us an isomorphism of $C([0,1])$-algebras $C^*(\G_\Lambda \times [0,1], \omega) \cong C([0,1]) \otimes C^*(\Lambda)$.   

For each $t \in [0,1]$, let $\pi^t: C^*(\Lambda, c_t) \to C^*(\G_\Lambda, \sigma_{c_t})$ denote the isomorphism described in Theorem 6.7 of \cite{kps-twisted}.  Let $\pi: C^*(\Lambda) \to C^*(\G_\Lambda)$ denote the equivalent isomorphism  for the case of a trivial cocycle $c$.  For each $n \in \Z^k$, write $U^t_n$ for the unitary $U_n^t: A_n \to A_n$ associated to the cocycle $c_t$ as above.  Setting
\[\Psi_t := \pi^t \circ \Ad U_*^t \circ \pi\inv\]
consequently gives an isomorphism of $C^*$-algebras $\Psi_t: C^*(\G_\Lambda) \to C^*(\G_\Lambda, \sigma_{c_t})$.

We claim that $\Psi := \{\Psi_t\}_{t \in [0,1]}$ defines an isomorphism of $C([0,1])$-algebras
\[\Psi: C^*(\G_\Lambda \times [0,1]) \to C^*(\G_\Lambda \times [0,1], \omega).\]  In order to prove this assertion, we begin by writing down an explicit formula for $\Psi_t$ on the characteristic functions $1_{Z(\mu,\nu)}\in C_c(\G_\Lambda)$ where $Z(\mu, \nu) \in \mathcal{P}$, where $\mathcal{P}$ is the partition of $\G_\Lambda$ described in Lemma 6.6 of \cite{kps-twisted}. 


Recall that the value of $\sigma_{c_t}(a,b)$ depends only on the sets $Z(\mu, \nu) \in \mathcal{P}$ containing the points $a,b,$ and $ab$ in $\G_\Lambda$.  Moreover, the proof of \cite{kps-twisted} Theorem 6.7 establishes that, if $1_{Z(\mu,\nu)}$ denotes the characteristic function on 
$Z(\mu, \nu) \subseteq \G_\Lambda$, and we write $a \in Z(\mu,\nu)$ as $a = bd$ where $b \in Z(\mu, s(\mu)),\ d  \in Z(s(\nu), \nu),$
\begin{align*} \pi^t(s_\mu s_\nu^*)(a) &= 1_{Z(\mu,\nu)}(a) \sigma_{c_t}(b,d) \overline{\sigma_{c_t}(d\inv, d)} = 1_{Z(\mu,\nu)}(a) \overline{\sigma_{c_t}(bd, d\inv)}.
\end{align*}
Moreover, we have $Z(\mu, s(\mu)) \in \mathcal{P} \ \forall \ \mu \in \Lambda$ by Lemma 6.6 of \cite{kps-twisted}.
If we also have $Z(\mu,\nu) \in \mathcal{P}$, then the elements $\alpha, \beta, \gamma$ in the formula for $\sigma_{c_t}(bd, d\inv)$ given in Definition \ref{sigma-c} are all units, so 
for any $t$, $\sigma_{c_t}(bd, d\inv) = 1$
 by our hypothesis that any cocycle $c$ satisfy $c(\lambda, s(\lambda)) = c(r(\lambda), \lambda) = 1$. Thus, 
 \[ Z(\mu, \nu) \in \mathcal{P} \ \Rightarrow\  \pi^t(s_\mu s_\nu^*) = 1_{Z(\mu,\nu)} \ \Rightarrow \ \Psi_t(1_{Z(\mu,\nu)}) = \kappa_t(\mu) \overline{\kappa_t(\nu)} 1_{Z(\mu,\nu)}.\]

Now, observe that each $f \in C_c(\G_\Lambda \times [0,1])$ can be written as a finite sum $f(a,t) = \sum_{i \in N} f_i(a,t),$
where, for all $i$, $f_i \in C(Z(\mu_i, \nu_i) \times [0,1])$ and  $Z(\mu_i, \nu_i) \in \mathcal{P}$.
Consequently, on $C_c(\G_\Lambda \times [0,1])$, our map $\Psi$ becomes
\begin{equation}
\Psi\left( \sum_{i \in N} f_i  \right)(a,t) = \sum_{i\in N} \Psi_t(f_i(\cdot, t))(a) =  \sum_{i \in N} \kappa_t(\mu_i) \overline{\kappa_t(\nu_i)} f_i(a,t);
\label{psi}
\end{equation}
the fact that all the sums are finite implies that $\Psi$ takes $C_c(\G_\Lambda \times [0,1]) $ onto $C_c(\G_\Lambda \times [0,1])$. 
  
%
 Since $\Psi$ is evidently $C([0,1])$-linear 
 and is a $*$-isomorphism in each fiber, 
 Proposition C.10 of \cite{xprod} 
tells us that $\Psi$  is norm-preserving.  Moreover, $\Psi$ is a $*$-homomorphism since the operations in $C_c(\G_\Lambda \times [0,1])$ preserve the fiber over $t \in [0,1]$, and each $\Psi_t$ is a $*$-homomorphism.

%

In other words, $\Psi$ extends to an isomorphism of $C([0,1])$-algebras 
\[\Psi: C^*(\G_\Lambda \times [0,1]) \cong C^*(\G_\Lambda \times [0,1], \omega).\] A straightforward check will establish that the identity map on $C_c(\G_\Lambda \times [0,1])$  induces an isomorphism $ id: C^*(\G_\Lambda \times [0,1])\to  C([0,1], C^*(\G_\Lambda))$ of $C([0,1])$-algebras; the isomorphism $C^*(\G_\Lambda) \cong C^*(\Lambda)$ 
of \cite{kp} Corollary 3.5(i) now finishes the proof.
%
%
%
%
%
 \end{proof}

\begin{rmk}
\label{phi}
Note that $\Psi$ induces an isomorphism $\Phi: C([0,1]) \otimes C^*(\Lambda) \to C^*(\G_\Lambda \times [0,1], \omega)$ as follows. 
If $Z(\mu,\nu) \in \mathcal{P}$ and $f \in C([0,1])$, then 
\begin{equation}
\Phi(f \otimes s_\mu s_\nu^*)(x,t) = f(t) 1_{Z(\mu,\nu)}(x) \kappa_t(\mu) \overline{\kappa_t(\nu)}.
\label{eqn-phi}
\end{equation}
\end{rmk}

\begin{rmk} Since evaluation at $t \in [0,1]$ induces a homotopy equivalence between $C([0,1], C^*(\Lambda))$ and $C^*(\Lambda)$, the isomorphism established in the previous Proposition implies  that evaluation at $t$ also induces a homotopy equivalence between  $C^*(\G_\Lambda \times [0,1], \omega)$ and its fiber algebra $C^*(\G_\Lambda, \sigma_{c_t})$ when $d = \delta b$.
\label{htopy-equiv}
\end{rmk}

\subsection{Proof of Theorem \ref{main}}
\label{5}

To leverage Proposition \ref{trivial} into the proof of Theorem \ref{main}, we will use the skew-product $k$-graphs $\Lambda \times_d \Z^k$:
\begin{defn}[\cite{kp} Definition 5.1]
Given a $k$-graph $(\Lambda,d)$, the \emph{skew-product $k$-graph $\Lambda \times_d \Z^k$} is the set $\Lambda \times \Z^k$, with the structure maps 
\[
r(\lambda, n) = (r(\lambda), n); \quad 
s(\lambda, n) = (s(\lambda), n + d(\lambda)); \quad 
d(\lambda, n) = d(\lambda),
\]
and multiplication given by $(\lambda, n)(\mu, n + d(\lambda)) = (\lambda \mu, n)$ for $(\lambda, \mu) \in \Lambda^{*2}$.
\end{defn}

Observe that the function $b: (\Lambda \times_d \Z^k)\z = \Lambda\z \times \Z^k \to \Z^k$ given by $b(v, n) = n$ satisfies $\delta b = d$ on $\Lambda \times_d \Z^k$. Moreover, if $\Lambda$ is row-finite and source-free, then so is $\Lambda \times_d \Z^k$.

We can now complete the proof of Theorem \ref{main}.

\begin{proof}[Proof of Theorem \ref{main}:]
Let $\phi: \Lambda \times_d \Z^k \to \Lambda$ be the projection onto the first coordinate: $\phi(\lambda, n) = \lambda$.  A cocycle $c$ on  $\Lambda$ induces a cocycle $c \circ \phi$ on the skew product $k$-graph $\Lambda \times_d \Z^k$:
\[c \circ \phi\left( (\lambda, n), (\mu, n + d(\lambda) ) \right) := c(\lambda, \mu)\]
whenever $(\lambda, \mu) \in \Lambda^{*2}$.
Note that if $\{c_t\}_{t \in [0,1]}$ is a homotopy of cocycles on $\Lambda$ 
 then $\{c_t \circ \phi\}_t$ is also a homotopy of cocycles on $\Lambda \times_d \Z^k$.

If $\omega$ is the homotopy of cocycles on $\G_{\Lambda \times_d \Z^k}$ associated to the homotopy $\{c_t\}_{t \in [0,1]}$ of cocycles  on $\Lambda$, then Proposition \ref{trivial} tells us that
\[C^*(\G_{\Lambda \times_d \Z^k} \times [0,1], \omega) \cong C([0,1]) \otimes C^*(\Lambda \times_d \Z^k).\]

Now, we
define an action of $\Z^k$ on $C([0,1]) \otimes C^*(\Lambda \times_d \Z^k)$ by setting
\begin{equation}f \otimes s_{\lambda, n} \cdot m := f \otimes s_{\lambda, n+m}. \label{zk-action} \end{equation}
To see that this formula gives us a well-defined action of $\Z^k$ on $C([0,1]) \otimes C^*(\Lambda \times_d \Z^k)$, one checks first that for each $m \in \Z^k$,
$\{s_{\lambda, m+n}: \lambda \in \Lambda, n \in \Z^k\}$ is a collection of partial isometries satisfying the defining axioms (CK1)-(CK4) for $C^*(\Lambda \times_d \Z^k)$.
Consequently, the universal property of $C^*(\Lambda \times_d \Z^k)$ implies that for each fixed $m \in \Z^k$,   the map $s_{\lambda, n} \mapsto s_{\lambda, n+m}$ determines a $*$-homomorphism 
\[\alpha_m: C^*(\Lambda \times_d \Z^k) \to C^*(\Lambda \times_d \Z^k).\]
Each $\alpha_m$ is invertible with inverse $\alpha_{-m}$; it follows that $m \mapsto \alpha_m$ defines a group action of $\Z^k$ on $C^*(\Lambda \times_d \Z^k)$.  Thus, Equation \eqref{zk-action} describes a well-defined action $ id \otimes \alpha$ of $\Z^k$ on $C([0,1]) \otimes C^*(\Lambda \times_d \Z^k)$, given by $m \mapsto id \otimes \alpha_m$.  The fact that the degree map on $\Lambda \times_d \Z^k$ is a coboundary now allows us to combine the action $id \otimes \alpha$ with the isomorphism  $\Phi: C([0,1]) \otimes C^*(\Lambda \times_d \Z^k) \to C^*(\G_{\Lambda \times_d \Z^k} \times [0,1], \omega)$ 
of Remark \ref{phi} to obtain an action $\beta$ of $\Z^k$ on $C^*(\G_{\Lambda \times_d \Z^k} \times [0,1], \omega)$:
\[\beta_n \left(\Phi(f \otimes s_{\mu, m} s_{\nu, m + d(\mu)-d(\nu)}^*) \right):= \Phi(id \otimes \alpha_n(f \otimes s_{\mu, m} s_{\nu, m + d(\mu)-d(\nu)}^*)).\]

Moreover, since both $id \otimes \alpha$ and $\Phi$ (and hence $\beta$) fix $C([0,1])$ by construction,
 Lemma 5.3 of \cite{kps-kthy} tells us that the crossed product
\[C^*(\G_{\Lambda \times_d \Z^k} \times [0,1], \omega) \rtimes_\beta \Z^k \cong \left( C([0,1]) \otimes C^*(\Lambda \times_d \Z^k) \right) \rtimes_{id \otimes \alpha} \Z^k\]
is a $C([0,1])$-algebra with fiber $C^*(\G_{\Lambda \times_d \Z^k}, \sigma_{c_t \circ \phi}) \rtimes_{\beta_t} \Z^k$, where 
\begin{align*}
 (\beta_t)_n( \Phi_t(s_{\mu, m}s_{\nu, m + d(\mu)-d(\nu)}^*)) &= \Phi_t(\alpha_n(s_{(\mu,m)} s_{(\nu, m+ d(\mu) - d(\nu))}^*))\\
 &= \kappa_t(\mu) \overline{\kappa_t(\nu)}  1_{Z((\mu, m+n), (\nu, m+n+ d(\mu) - d(\nu)))}
\end{align*}
whenever $Z((\mu, m+n), (\nu, m+n+ d(\mu) - d(\nu)))$ is in the partition $\mathcal{P}$ of $\G_{\Lambda \times_d \Z^k}$ that we used in the proof of Proposition \ref{trivial}.

Recall that we have a homotopy equivalence $q_t: C^*(\G_{\Lambda \times_d \Z^k} \times [0,1], \omega) \to C^*(\G_{\Lambda \times_d \Z^k}, \sigma_{c_t})$. 
A computation will show that $q_t$ is equivariant with respect to the actions $\beta, \beta_t$ of $\Z^k$; thus, Theorem 5.1 of \cite{kps-kthy} tells us that 
\begin{equation}K_*(C^*(\G_{\Lambda \times_d \Z^k} \times [0,1], \omega) \rtimes_\beta \Z^k) \cong K_*(C^*(\G_{\Lambda \times_d \Z^k} , \sigma_{ c_t \circ \phi}) \rtimes_{\beta_t} \Z^k).\label{gpoid-kthy}
\end{equation}

Thanks to Lemma 5.2 of \cite{kps-kthy}, we know that  
$C^*(\Lambda \times_d \Z^k, c_t \circ \phi) \rtimes_{lt} \Z^k \sim_{ME} C^*(\Lambda, c_t)$,
where $lt_m (s_{\lambda, n}) = s_{\lambda, n+m}$.
To make use of this result, we need to show that $\beta_t $ induces the action $lt$ on $C^*(\Lambda \times_d \Z^k, c_t \circ \phi)$.

Recall from 
the proof of Proposition \ref{trivial}  that $\pi^t(s_{\lambda, m}) = 1_{Z((\lambda, m), (s(\lambda), m + s(\lambda)))}$, since $Z((\lambda, m), (s(\lambda), m + s(\lambda))) \in \mathcal{P}$ always.  Observe that
\begin{equation}C^*(\G_{\Lambda \times_d \Z^k}, \sigma_{ c_t \circ \phi}) \rtimes_{\beta_t} \Z^k \cong C^*(\Lambda \times_d \Z^k,  c_t \circ \phi) \rtimes_{\gamma_t} \Z^k, \text{ where}
\label{actions}\end{equation}
\begin{align*}
(\gamma_t)_n(s_{\lambda, m}) &:= (\pi^t)\inv((\beta_t)_n(\pi^t(s_{\lambda, m}))) = (\pi^t)\inv( \beta_t)_n(1_{Z((\lambda, m), (s(\lambda), m))}) \\
&= (\pi^t)\inv( \beta_t)_n(\Phi_t(\overline{\kappa_t(\lambda)} s_{(\lambda, m)})) = (\pi^t)\inv \left( \Phi_t(\alpha_n(\overline{\kappa_t(\lambda)} s_{\lambda, m} ))  \right) \\
&= (\pi^t)\inv \left( \Phi_t(\overline{\kappa_t(\lambda)} s_{\lambda, m+n}) \right)= (\pi^t)\inv \left(1_{Z((\lambda, m+n), (s(\lambda), m+n))} \right) \\
&= s_{\lambda, m+n}.
\end{align*} 
It follows that the action $(\gamma_t)$ induced by $\beta_t$ agrees with $lt$, as desired.
 Now, the Morita equivalence of Lemma 5.2 of \cite{kps-kthy} and Equation \eqref{actions} tell us that 
\begin{equation}C^*(\G_{\Lambda \times_d \Z^k}, \sigma_{ c_t \circ \phi}) \rtimes_{\beta_t} \Z^k \sim_{ME} C^*(\Lambda, c_t).\label{me}
\end{equation}
Combining Equations \eqref{gpoid-kthy} and \eqref{me} now yields 
\[K_*(C^*(\Lambda, c_t)) \cong  K_*(C^*(\G_{\Lambda \times_d \Z^k} \times [0,1], \omega) \rtimes_\beta \Z^k)\]
for any $t \in [0,1]$.  It follows that, if $\{c_t\}_{t\in[0,1]}$ is a homotopy of cocycles on a row-finite $k$-graph $\Lambda$ with no sources, then for any $s, t \in [0,1]$, 
\[K_*(C^*(\Lambda, c_t)) \cong K_*(C^*(\Lambda, c_s)).\]

It remains to show that this isomorphism preserves the $K$-theory class of each  vertex projection $s_v$.  Essentially, this follows because the cocycles $c_t$, and thus the functions $\kappa_t$, are all trivial on any $v \in \text{Obj}(\Lambda)$.

To be precise, let $v \in \text{Obj}(\Lambda)$ and define $f_v \in C_c(\Z^k, C_c(\G_{\Lambda \times_d \Z^k} \times [0,1]))) \subseteq C^*(\G_{\Lambda \times_d \Z^k} \times [0,1], \omega) \rtimes_\beta \Z^k$  by 
\[f_v(n)(a, t) = \left\{ \begin{array}{lc} 1, & a \in Z_{(v,0), (v,0)} \text{ and } n=0 \\ 0,& \text{else.} \end{array} \right. \]
Then the projection $q_t \rtimes id(f_v)$ of $f_v$ onto the fiber algebra $C^*(\G_{\Lambda \times_d \Z^k}, \omega_t) \rtimes_{\beta_t} \Z^k$ is independent of the choice of $t \in [0,1]$:
\[q_t \rtimes id(f_v)(n)( a) =  \left\{ \begin{array}{lc} 1, & a \in Z_{(v,0), (v,0)} \text{ and } n=0 \\ 0,& \text{else} \end{array} \right. \]
for any $t \in [0,1]$.
Moreover, the isomorphism 
$\Phi_t: C^*(\Lambda \times_d \Z^k, c_t \circ \phi) \to C^*(\G_{\Lambda \times_d \Z^k}, \sigma_{c_t \circ \phi})$
of Remark \ref{phi} satisfies
\begin{equation}
\Phi_t \rtimes id (j(s_{(v,0)})) = q_t\rtimes id(f_v),
\label{idk2}
\end{equation}
where $j: C^*(\Lambda \times_d \Z^k, c_t \circ \phi) \to C^*(\Lambda \times_d \Z^k, c_t \circ \phi) \rtimes_{lt} \Z^k$ is the canonical embedding of $C^*(\Lambda \times_d \Z^k, c_t \circ \phi)$ into the crossed product.

The fact that the Morita equivalence $C^*(\Lambda, c_t) \sim_{ME} C^*(\Lambda \times_d \Z^k, c_t \circ \phi) \rtimes_{lt} \Z^k$
takes $s_v \in C^*(\Lambda, c_t)$ to $j(s_{(v,0)})$ (cf.~Lemma 5.2 in \cite{kps-kthy}) thus implies that our $K$-theoretic isomorphism
$ K_*(C^*(\G_{\Lambda \times_d \Z^k} \times [0,1], \omega) \rtimes_\beta \Z^k) \to K_*(C^*(\Lambda, c_t)),$
which is given by the composition of the Morita equivalence \eqref{me} with the $*$-homomorphism 
\begin{align*}
q_t \rtimes id: C^*(\G_{\Lambda \times_d \Z^k} \times [0,1], \omega) \rtimes_\beta \Z^k & \to C^*(\G_{\Lambda \times_d \Z^k}, \omega_t) \rtimes_{\beta_t} \Z^k \\
& \cong C^*(\Lambda \times_d \Z^k, c_t \circ \phi) \rtimes_{\gamma_t} \Z^k,
\end{align*}
takes $[f_v]$ to $[s_v]$ for any $v \in \text{Obj}(\Lambda)$ and any $t \in [0,1]$.  Consequently, the isomorphism $ K_*(C^*(\Lambda, c_t)) \cong K_*(C^*(\Lambda, c_s))$
preserves the class of $s_v$, as claimed.  This finishes
 the proof of Theorem \ref{main}. 
 \end{proof}

\begin{rmk}
It's tempting to think that since $C^*(\Lambda \times_d \Z^k, c \circ \phi) \cong C^*(\Lambda \times_d \Z^k)$ and $C^*(\Lambda, c) \sim_{ME} C^*(\Lambda \times_d \Z^k, c \circ \phi) \rtimes_{lt} \Z^k$ for any cocycle $c$ on $\Lambda$, any two twisted $k$-graph $C^*$-algebras should be Morita equivalent.  This statement is false, however (the rotation algebras provide a counterexample).  The flaw lies in the fact that the isomorphism $\Ad U_*: C^*(\Lambda \times_d \Z^k, c \circ \phi) \to C^*(\Lambda \times_d \Z^k)$ is not equivariant with respect to the left-translation action of $\Z^k$, 
so the isomorphism 
\[C^*(\Lambda \times_d \Z^k, c \circ \phi) \cong C^*(\Lambda \times_d \Z^k)\]
does not pass to an isomorphism $C^*(\Lambda \times_d \Z^k, c \circ \phi) \rtimes_{lt} \Z^k \to C^*(\Lambda \times_d \Z^k) \rtimes_{lt} \Z^k$.  In other words, a $K$-theoretic equivalence of twisted $k$-graph $C^*$-algebras is the best result we can hope for in general.
\end{rmk}

\section{Future work} 
The standing hypotheses of this paper, that our $k$-graphs be row-finite and source-free, are slightly more restrictive than the current standard for $k$-graphs.  Thus, we would like  to extend Theorem \ref{main} to apply to all finitely aligned $k$-graphs. Finitely aligned $k$-graphs were introduced in \cite{raeburn-sims, rsy2}, and it seems that they constitute the largest class of $k$-graphs to which one can profitably associate a $C^*$-algebra.
  However, the Kumjian-Pask construction of a groupoid $\G_\Lambda$ associated to a $k$-graph $\Lambda$, which we described  in Section \ref{2} and which we use throughout the proof of Theorem \ref{main}, only works when $\Lambda$ is row-finite and source-free.  In \cite{fmy}, Farthing, Muhly, and Yeend provide an alternate construction of a groupoid $\G$ which can be associated to an arbitrary finitely-aligned $k$-graph, and we hope that this approach will allow us to apply groupoid results such as Theorem \ref{fiber-alg} to study the effect on $K$-theory of homotopies of cocycles for finitely-aligned $k$-graphs.

%

\bibliographystyle{amsplain}
\bibliography{eagbib}
\end{document}